\documentclass[11pt, reqno]{amsart}
\usepackage[margin=3.6cm]{geometry}

\usepackage{amsmath, amsfonts, amsthm, amssymb, dsfont}
\usepackage{exscale}
\usepackage{cite}
\usepackage{epsfig}
\usepackage{amscd}
\usepackage{graphics}
\usepackage{color}

\renewcommand{\geq}{\geqslant}
\renewcommand{\leq}{\leqslant}

\newtheorem{theorem}{Theorem}

\newtheorem{proposition}{Proposition}[section]

\newtheorem{lemma}[proposition]{Lemma}
\newtheorem*{main-theorem}{Main Theorem}
\newtheorem*{theorem*}{Theorem}

\theoremstyle{definition}

\newtheorem{remark}[proposition]{Remark}

\newtheorem*{remark*}{Remark}

\numberwithin{equation}{section}

\def\phi{\varphi}

\def\reals{{\mathbb R}}

\def\O{{\mathcal O}}

\def\phi{\varphi}

\def\be{\begin{eqnarray*}}
\def\ee{\end{eqnarray*}}
\def\ben{\begin{eqnarray}}
\def\een{\end{eqnarray}}

\def\L2R{L_{\text{Rest}}^2}

\def\11{\mathds{1}}

\def\L2c{L^2_{\text{comp}}}

\def\tDelta{\widetilde{\Delta}}

\def\tS{\widetilde{S}}

\def\Vol{\text{Vol}}

\def\p{\partial}

\def\bu{\bar{u}}
\def\bv{\bar{v}}

\begin{document}

\title[Neumann data on simplices]{Equidistribution of 
  Neumann data mass on simplices and a simple inverse problem}


\author[H. Christianson]{Hans Christianson}
\address[H. Christianson]{ Department of Mathematics, University of North Carolina.\medskip}
 \email{hans@math.unc.edu}



%
%
%
\begin{abstract}
  In this paper we study the behaviour of the Neumann data 
of Dirichlet eigenfunctions on simplices.   We prove that the  $L^2$
norm of the (semi-classical) Neumann data on each face is equal to
$2/n$ times the $(n-1)$-dimensional volume of the face divided by the
volume of the simplex.
This is a  generalization of \cite{Chr-tri} to higher
dimensions.  Again it is {\it not} an asymptotic, but an exact formula.
  The proof is by simple integrations by parts and 
   linear algebra.

  We also consider the following inverse problem: do the {\it norms}
  of the Neumann data on a simplex determine a constant coefficient elliptic
  operator?  The answer is yes in dimension 2 and no in higher dimensions.

\end{abstract}

\maketitle

\section{Introduction}
In this paper we extend the results of \cite{Chr-tri} on triangles to
simplices, which are the higher dimensional analogues of triangles.
The proof has many similarities but involves more linear algebra and
elementary geometry.   We have chosen to separate the two proofs in
order to make the paper about triangles simple and clean.  We also
have added to this paper some applications to rudimentary inverse problems.  

Let $T \subset \reals^n$ be an $n$ dimensional (non-degenerate) simplex with faces
$G_0, \ldots G_{n}$.  We consider the Dirichlet eigenfunction
problem on $T$:
\begin{equation}
  \label{E:ef}
  \begin{cases}
    -h^2 \Delta u = u \text{ in }T, \\
    u|_{\p T} = 0.
  \end{cases}
  \end{equation}
The semiclassical parameter $h>0$ denotes the (inverse of) the
eigenvalues hence takes values in a discrete set.  We assume that the
eigenfunctions are normalized: $\| u \|_{L^2 (T) } = 1$.  Our main
result, similarly to in \cite{Chr-tri} is that the Neumann data on
each face of the simplex is proportional to the volume of the face.



\begin{theorem}
  \label{T:main}
  Let $T \subset \reals^n$ be a non-degenerate simplex with faces
  $G_0, G_1, \ldots , G_{n}$ and 
suppose $u$ solves \eqref{E:ef}.

  Then the (semi-classical) Neumann data on each of the the boundary faces satisfies
  \begin{equation}
    \label{E:main-theorem}
  \int_{G_j} | h \p_\omega u |^2 dS_j = \frac{2 \text{Vol}_{n-1}(G_j)}{n\text{Vol}_n(T)}.
  \end{equation}
  Here 
 $h \p_\omega$ is the semi-classical normal derivative on $\p T$,
$dS_j$ is the surface  measure on $G_j$,  $\text{Vol}_n(T)$ is the
  volume of the simplex $T$, and $\text{Vol}_{n-1} (G_j)$ is the $n-1$
  dimensional induced volume of $G_j$.

\end{theorem}



\begin{remark}

  As in \cite{Chr-tri}, we are calling this ``equidistribution'' of Neumann mass since it
  says that the Neumann data has mass proportional to the
  $(n-1)$-dimensional volume of the face to which it is restricted.

The proportionality constant in \eqref{E:main-theorem} depends in a seemingly
non-obvious way on the dimension $n$.  However, it turns out this is
the right dimensional constant in the case of the Cauchy data for
quantum ergodic eigenfunctions restricted to a hypersurface, and
indeed also for the boundary data quantum ergodic restriction theorems
in the original studies \cite{GeLe-qe,HaZe}.  One of
the original motivations for the present paper was to see if one could
isolate the mass of the Dirichlet vs. Neumann data of quantum ergodic
eigenfunctions restricted to an interior simplex hypersurface in the Cauchy data
restriction theorem in \cite{CTZ-1}.
Unfortunately this does not help, and the present paper and
\cite{Chr-tri} do not preclude the
possibility of quantum ergodic eigenfunctions having $o(1)$ (in $L^2$)
restrictions to the boundary of an interior  simplex.  See below for a
brief history.

 A statement such as Theorem \ref{T:main}
 is false in general for other  polygonal domains.  It is clearly
 false in the case of a square, as discussed in \cite{Chr-tri}, as
 well as for a rectangular parallelepiped in any dimension by looking
 at Fourier series. 

\end{remark}

\subsection{Brief History}

  Previous results on restrictions to hypersurfaces primarily focused on upper bounds.
  Burq-G\'erard-Tzvetkov \cite{BGT-erest} 
  give an 
  upper
  bound of the norm (squared) of the restrictions of eigenfunctions,
  of order $\O(h^{-1/2})$.  
  In
  the author's paper with Hassell-Toth \cite{CHT-ND}, an upper bound of $\O(1)$ was
  proved for (semi-classical) Neumann data restricted to arbitrary
  co-dimension $1$ hypersurfaces in any dimension.
  Both of these estimates are shown to be sharp, so this gives a lower (and
  upper) bound for {\it some} eigenfunctions.

  In the case of quantum ergodic eigenfunctions, a little more is
  known.  G\'erard-Leichtnam \cite{GeLe-qe} and Hassell-Zelditch
  \cite{HaZe} give asymptotic formulae for (a density one subsequence of) the Neumann (respectively
  Dirichlet) boundary data of Dirichlet (respectively Neumann) quantum ergodic eigenfunctions. 
  That means that there is a lower bound, and
  explicit local asymptotic formula in this special case, at least for
  most of the eigenfunctions.  Similar statements were proved for
  interior hypersurfaces in \cite{ToZe-1, ToZe-2, CTZ-1}.  However,
  for an interior hypersurface, it seems an intractible problem to
  separate the behaviour of the Dirichlet or Neumann data, or a sparse
  subsequence must be removed.  
This again gives lower bounds on the norms of the Dirichlet or Neumann data for {\it some} of the eigenfunctions.

\subsection*{Acknowledgements}
The work in this paper is supported in part by NSF grant DMS-1500812.

  \section{The standard simplex in $\reals^3$}
  \label{S:standard}
In this section we prove the theorem for  the standard simplex in
dimension $3$ as it
is simple to see how the proof works in this case.  In Section 
\ref{S:main} we prove the general result.

Let $p_0 = (0,0,0)$, $p_1 = (1,0,0)$, $p_2 = (0,1,0)$, and $p_3 =
(0,0,1)$.  The standard simplex is given by all convex combinations of
these vectors:
\[
T = \left\{ \sum_{j = 0}^3 t_j p_j : \sum_{j = 0}^3 t_j = 1, \text{ and }
t_j \geq 0 \right\}.
\]
That is, $T$ is the four sided solid with the $p_j$ and $0$ at the corners.

We use $(x_1,x_2,x_3)$ as the standard rectangular coordinates in
$\reals^3$.  Let us denote $F_1$ denote the face in the $(x_2, x_3)$
plane (where $x_1 = 0$), $F_2$ the face where $x_2 = 0$, $F_3$ the face where $x_3 = 0$,
and $F_4$ the remaining face.  Then the unit normals are $\nu_j = -e_j$, $j =
1,2,3$ and $\nu_4 = (3)^{-1/2} (1,1,1)$ respectively, where $e_j$ are the standard
basis vectors pointing in the direction of $x_j$ respectively.  Then
the statement of the theorem involves the quantities $|\nu_j \cdot h\p
u|$ restricted to their respective faces.

Let us denote by $X$ the vector field
\[
X = (x_1 + m_1) \p_{x_1} + (x_2 + m_2) \p_{x_2} + (x_3 + m_3)
\p_{x_3},
\]
where the $m_j$s are parameters independent of $x$.  
A simple computation yields that $[-h^2 \Delta -1, X] = - 2 h^2
\Delta$.  Then the eigenfunction equation \eqref{E:ef} and an application of Green's formula gives
\begin{align*}
  \int_T & ([-h^2 \Delta -1, X ] u) \bu dV \\
  & = \int_T ((-h^2 \Delta -1)X u ) \bu dV \\
  & = \int_{\p T } (-h \p_\nu hXu) \bu dS + \int_{\p T} (hXu)( h \p_\nu
  \bu ) dS,
\end{align*}
or
\begin{align}
2 & = 2 \int_T | u |^2 dV \notag \\
& =   -2 \int_T (h^2 \Delta u) \bu dV \notag \\
&  =  \int_{\p T } (-h \p_\nu hXu) \bu dS + \int_{\p T} (hXu)( h \p_\nu
\bu ) dS \\
& =\int_{\p T} (hXu)( h \p_\nu
\bu ) dS
,    \label{E:master-0}
\end{align}
since we have assumed Dirichlet boundary conditions.

Let us break the analysis into the four different faces.
On $F_1$, we have
\begin{align*}
\int_{F_1} & 
(h X u ) ( h \p_\nu \bu ) dS \\
& = \int_{F_1} (((x_1 + m_1 ) h \p_{x_1} + (x_2 + m_2 ) h \p_{x_2} +
(x_3 + m_3 ) h \p_{x_3} ) u ) \bu dS_1 \\
& = -m_1 \int_{F_1} | h \p_{\nu_1} u |^2 dS_1,
\end{align*}
since $h \p_{x_1} = - h \p_{\nu_1}$ and $h \p_{x_j}$ is tangential
  when $j = 2,3$.
  Similarly, for $j = 2,3$ we have
  \[
  \int_{F_j} (hX u ) (h \p_{\nu_j} \bu ) dS_j = -m_j \int_{F_j} | h
  \p_{\nu_j} |^2 d S_j.
  \]

  On $F_4$ we need to be a little bit more careful.  The points on
  $F_4$ all satisfy $x_1 + x_2 + x_3 = 1$ since the normal is parallel
  to $(1,1,1)$.  
The normal derivative is $h\p_{\nu_4} = 3^{-1/2} ( h \p_{x_1} + h
\p_{x_2} + h \p_{x_3} )$, and 
 the tangent vectors are all linear combinations of
 $e_3 - e_1 = (-1,0,0)$ and $e_2 - e_1 = (-1,1,0)$, so that
 \[
 \p_{x_j} = 3^{-1/2} \p_{\nu_4}
 \]
 for $j = 1,2,3$.  Hence
 \begin{align*}
   \int_{F_4} & (h X u ) h \p_{\nu_4} \bu dS_4 \\
 &  = \int_{F_4} ( ( (x_1 + m_1 ) h \p_{x_1} + (x_2 + m_2 ) h
   \p_{x_2}+(x_3 + m_3 ) h \p_{x_3} ) u ) h \p_{\nu_4} \bu d S_4 \\
   & =  (3)^{-1/2} \int_{F_4} ( ( (x_1 + m_1 )  + (x_2 + m_2 ) +(x_3 +
   m_3 )  ) h \p_{\nu_4} u ) h \p_{\nu_4} \bu d S_4
   \\
   & =  3^{-1/2} (1 +  m_1 + m_2 + m_3 )  \int_{F_4} | h \p_{\nu_4} u |^2
   dS_4.
 \end{align*}
 Summing up, we have
 \begin{align}
   2 & = -m_1 \int_{F_1} | h \p_{\nu_1} u |^2 dS_1  - m_2 \int_{F_2} |
   h \p_{\nu_2} u |^2 dS_2 -m_3 \int_{F_3} | h \p_{\nu_3} u |^2 dS_3
   \notag 
   \\
   & + 3^{-1/2} (1 + m_1 + m_2 + m_3 ) \int_{F_4} | h \p_{\nu_4} u |^2
   dS_4. \label{E:master-1}
   \end{align}

 Now if $m_j = 0$ for $j = 1,2,3$, using \eqref{E:master-1} we have
 \[
 2 = (3)^{-1/2} \int_{F_4} | h \p_{\nu_4} u |^2
 dS_4,
 \]
 so that
 \[
 \int_{F_4} | h \p_{\nu_4} u |^2 dS_4 = 3^{1/2} \cdot  2   .
 \]
 We know that $\text{Vol}_3 (T) = 1/3! = 1/6$.  The cross product
 computes the area of the parallelogram, which is twice the area of
 the triangle, so that 
 tells us that
 \begin{align*}
   \text{Vol}_2 (F_4) & = |(-1,1,0) \times (-1,0,1) |/2 \\
   & = \sqrt{3} /2.
 \end{align*}
 Hence
 \begin{align*}
   \int_{F_4} | h \p_{\nu_4} u |^2 dS_4 & = {2}\cdot {  3^{1/2}} \\
   & =  4 ({\sqrt{3} / 2 }) \\
   & = (2/3) \left( \frac{2 \cdot 3^{1/2}}{1/6} \right) \\
   & = \frac{2 \text{Vol}_{2}(F_4)}{n\text{Vol}_3(T)}.
 \end{align*}

 For $j = 1,2,3$ we have
 \[
 \text{Vol}_{2} (F_j ) = 1/2.
 \]
 Differentiating \eqref{E:master-1} with respect to $m_j$, we have
 \[
 0 = - \int_{F_j} | h \p_{\nu_j} u |^2 d S_j + (3)^{-1/2} \int_{F_4} |
 h \p_{\nu_4} u |^2 d S_4 ,
 \]
 or
 \begin{align*}
   2 & = \int_{F_j }  | h \p_{\nu_j} u |^2 dS_j \\
   & = \left(\frac{2}{3} \right) \left( \frac{1/2}{1/6} \right) \\
   & = \left( \frac{2}{3} \right) \frac{ \text{Vol}_2 (F_j)
   }{\text{Vol}_3 (T)}.
   \end{align*}
 This proves the theorem for the standard simplex in dimension $3$.

\section{Proof of Theorem \ref{T:main}}
\label{S:main}

Let $p_1, \ldots, p_n$ be independent vectors in $\reals^n$, and let
$p_0 = (0, \ldots , 0 )$ denote the origin.  Then
\[
T = \left\{ \sum_0^n t_j p_j : \sum t_j = 1 \text{ and } t_j \geq 0 \right\}
\]
is a simplex.  If $p_j = e_j$ (standard rectangular basis vectors) for
each $j$, then we say $T$ is the standard simplex and denote it by
$T_0$.  

Since the $p_j$s are independent, the matrix
\[
A = \left[ \begin{array}{cccc} \vline & \vline & \cdots & \vline \\
    p_1 & p_2 & \cdots & p_n \\
    \vline & \vline & \cdots & \vline
  \end{array}
  \right]
\]
is invertible.  Let $B = A^{-1}$, and for $x \in \reals^n$ set
\[
y = B x.
\]
This transformation simply takes the simplex $T$ to the standard
simplex $T_0$.  Indeed, if $x = p_j$, then $Bx = e_j$.  Hence
\[
T_0 = \left\{ \sum t_j B p_j , \, \sum t_j = 1, \, t_j \geq 0
\forall j \right\}.
\]
We
pause briefly to point out that this change of variables induces a
volume element, so that
\[
\det(A) = n! \Vol (T).
\]
This is easily seen using the volume of the standard simplex is $1/n!$
and the Jacobian for a change of volume integral is $\det (A)$.

We lift the transformation to $T^*\reals^n$: for $\xi \in \reals^n$,
let $\eta = (B^{-1})^T \xi$.  Then since the symbol of the Laplacian
in $\reals^n$ is $\xi_1^2 + \ldots + \xi_n^2$, the symbol for the
Laplacian in our new coordinates is
\[
\xi^T \xi = \eta^T B B^T \eta.
\]
Set $\Gamma = B B^T$ and
\[
-h^2 \tDelta = - \sum \Gamma_{ij} \p_{y_i} \p_{y_j},
\]
the Laplacian in the $y$ coordinates on the standard simplex $T_0$.

For the eigenfunctions $u$ on $T$, let $v(y) = u (Ay)$ be the
eigenfunctions in the $y$ coordinates.  Since $-h^2 \tDelta$ is constant coefficient, the same commutator
argument can be used here.  Indeed, let
\[
Y = \sum (y_j + m_j) \p_{y_j},
\]
and a simple calculation gives
\[
  [-h^2 \tDelta -1 , Y ] = -2 h^2 \tDelta.
  \]
  Following the recipe in Section \ref{S:standard}, we have using $-h^2 \tDelta
  v = v$ and Green's formula
  \begin{align}
    2 \int_{T_0} |v|^2 dy  & = -2 \int_{T_0} (h^2 \tDelta v) \bv dy
    \notag \\
    & = 
        \int_{T_0}  ([-h^2 \tDelta -1 , Y ] v )\bv dy \notag \\
        & = \int_{ T_0} ((-h^2 \tDelta -1) Y v ) \bv dy \notag \\
& = \int_{T_0} ( (-(h\p)^T  B B^T h \p - 1)Y v) \bv dy \notag \\
& = \int_{T_0} (B B^T h \p  Yv ) \cdot ( h \p \bv ) dy  - \int_{T_0}
(Y v) \bv dy
\notag \\
& \quad  + \int_{\p T_0 } ( -  \nu^T  B B^T h \p (h Y v) ) \bv dy  
\notag \\
& = \int_{T_0} (B B^T h \p  Yv ) \cdot ( h \p \bv ) dy   - \int_{T_0}
(Y v) \bv dy
\notag 
\end{align}
since we have assumed Dirichlet boundary condtions.  Here $\nu$
denotes the unit outward normal and $dS$ denotes the induced surface measure.  Continuing,
\begin{align}
2 \int_{T_0} |v|^2 dy  & =\int_{T_0} (B B^T h \p  Yv ) \cdot ( h \p \bv ) dy   - \int_{T_0}
(Y v) \bv dy
\notag \\
& = \int_{T_0} (Yv) (- h \p^T B B^T h \p \bv ) dy   - \int_{T_0}
(Y v) \bv dy
\notag
\\
& \quad         +  \int_{\p T_0} (hY v ) (\nu^T B B^T h \p \bv) dS \\
& =  \int_{\p T_0} (hY v ) (\nu^T B B^T h \p \bv) dS
\label{E:master-2}
  \end{align}
since $\tDelta \bv = \bv$.  


We have changed variables to be on $T_0$ in order to make sure the
normal vectors are easy to compute.  
For $T_0$, let $F_j$ be the side where
$y_j = 0$, $1 \leq j \leq n$, and $F_0$ the remaining face.  Then for $1 \leq j \leq n$,
we have the outgoing normal vectors to  $F_j$ $\nu_j = - e_j$, where
the $e_j$ are the standard basis vectors.
For $F_0$, we have transformed to $T_0$ so that
\[
\nu_0  =  n^{-1/2} (1, \ldots, 1).
\]
Then the unit normal derivatives are
\[
h \p_{\nu_j} = - h \p_{y_j}
\]
for $1 \leq j \leq n$ and
\[
h \p_{\nu_0} = n^{-1/2} ( h \p_{x_1} + \ldots + h \p_{x_n} ).
\]
We are assuming  Dirichlet boundary conditions, so all of the
tangential derivatives of $v$ vanish.  That is, for $1 \leq j \leq n$,
\[
h  \p_\ell v = 0,
\]
except for $\ell =  j$.  We also have using symmetry that on $F_0$,
\[
h \p_{\nu_0} v = n^{1/2} h \p_{y_j} v
\]
for every $1 \leq j \leq n$.  We recall again that $y_j = 0$ on $F_j$
for $1 \leq j \leq n$ and on $F_0$ we have $y_1 + y_2 + \ldots + y_n =
1$.

Plugging these observations in to \eqref{E:master-2}, we have
\begin{align}
  2 \int_{T_0} |v|^2 dy &  =  \int_{\p T_0} (hY v ) (\nu^T B B^T h \p
  \bv) dS \notag \\
  & = \sum_{j = 1}^n \int_{F_j} \left( \left(\sum_\ell (m_\ell + y_\ell) h
  \p_{y_\ell} \right) v \right) ( \nu_j^T B B^T h \p
  \bv ) dS_j \notag \\
  & \quad + \int_{F_0} \left(\left(\sum_\ell (m_\ell + y_\ell) h
  \p_{y_\ell} \right) v\right) ( {\nu_0}^T B B^T h \p \bv ) dS_0 \notag \\
& = \sum_{j = 1}^n \int_{F_j} ( m_j h \p_{y_j} v) ( h \p_{\nu_j}
  \bv ) dS_j \notag \\
  & \quad + \int_{F_0} \left(\sum_{1}^n (n^{-1/2} (y_j + m_j ))h
  \p_{\nu_0} v \right) (\nu_0^T B B^T h \p \bv ) dS_0 \notag \\
  & = \sum_{j = 1}^n \int_{F_j} (- m_j h \p_{\nu_j} v) ( \nu_j^T B B^T
  h \p
  \bv ) dS_j \notag \\
  & \quad + \int_{F_0} n^{-1/2} (1 + m_1 + \ldots + m_n)(( h
  \p_{\nu_0} ) v) (\nu_0^T B B^T h \p \bv ) dS_0 \notag \\
  & = \sum_{j = 1}^n (-m_j) I_j +
 n^{-1/2} (1 + m_1 + \ldots + m_n) I_0, \label{E:master-3}
\end{align}
where for each $0 \leq j \leq n$
\[
I_j = \int_{F_j} ( h \p_{\nu_j} v) (\nu_j^T B B^T h \p \bv )  dS_j.
\]

Let us now compute the $I_j$s.  Using equation \eqref{E:master-3},
setting $m_j = 0$ for all $ 1\leq j \leq n$, we have
\[
I_0 = 2 n^{1/2}  \int_{T_0} |v|^2 dy.
\]
Differentiating equation \eqref{E:master-3} with respect to $m_j$
yields for $1 \leq j \leq n$
\[
I_j = n^{-1/2} I_0 = 2 \int_{T_0} |v|^2 dy.
\]

Now we must compute the $I_j$ in terms of the corresponding integrals
on the original simplex $T$.  We first observe that, since for $1 \leq
j \leq n$ we have $F_j \subset \{ y_j = 0 \}$, changing variables on
one of the boundary integrals induces the area of the
$(n-1)$-dimensional parallelepiped spanned by $p_1, p_2, \ldots,
p_{j-1}, p_{j+1} , \ldots , p_n$.  Denote this parallelepiped
$\Gamma_j$, and observe that 
\[
\Vol_{n-1} \Gamma_j = (n-1)! \Vol_{n-1} (G_j),
\]
where $G_j$ is the $(n-1)$-dimensional simplex spanned by $p_1, p_2, \ldots,
p_{j-1}, p_{j+1} , \ldots , p_n$.

For $F_0$, our area element is $n^{1/2} dy$, so changing variables in
the integral over $F_0$ induces the area of the parallelepiped spanned
by $p_1, p_2 - p_1, p_3 - p_1, \ldots , p_n - p_1$ divided by
$n^{1/2}$.  Denote this parallelepiped by $\Gamma_0$, and again we
have
\[
\Vol_{n-1} (\Gamma_0) = (n-1)! \Vol_{n-1} (G_0).
\]

We now need to compute the integrand inside of each $I_j$ in terms of
the corresponding normal derivatives on $G_j$ of $u$.

We first observe that on $F_j$, for $1 \leq j \leq n$, $h \p_{y_\ell} \bv = 0$ for
$\ell \neq j$, so that the semiclassical gradient can be written 
\[
h \p_y v|_{F_j} = e_jh \p_{y_j} v|_{F_j} = \nu_j h \p_{\nu_j} v |_{F_j}.
\]
Similarly, for $j = 0$, we have on $F_0$
\begin{align*}
h \p v & = 
\left[ \begin{array}{c}
    h \p_{y_1} \\
    \vdots \\
    h \p_{y_n}
  \end{array}
  \right] v \\& = 
n^{-1/2} \left[ \begin{array}{c}
    1 \\
    \vdots \\
    1
  \end{array}
  \right] h \p_{\nu_0} v     \\
& = \nu_0 h \p_{\nu_0} v 
.
\end{align*}

Now for each $j$, let $\omega_j$ be the unit outward normal on $G_j$.  
We know for each $j$ on the face $G_j$
\begin{align*}
h\p_{\omega_j}u|_{G_j} & = \omega_j^T h \p_x u|_{G_j} \\
& = \omega_j^T  B^T h \p_y v|_{F_j}\\
& = (B \omega_j)^T h \p_y v |_{F_j}\\
& = (B \omega_j)^T \nu_j h \p_{\nu_j} v|_{F_j}\\
& = (\omega_j^T B^T \nu_j) h \p_{\nu_j} v|_{F_j} ,
\end{align*}
so that
\[
h \p_{\nu_j} v|_{F_j} = (\omega_j^T B^T \nu_j)^{-1} h \p_{\omega_j} u |_{G_j}
\]
written in the $y$ and $x$ coordinates respectively.

On the other hand, we have $h \p_x = B^T h \p_y$, so that
\[
\nu_j^T B h \p_xu = \nu_j^T B B^T h \p_y v.
\]
The left hand side is zero except for the projection on to the $\omega_j$,
so that on each $G_j$ we have 
\begin{align}
\nu_j^T B h \p_x u & = (\nu_j^T B \omega_j) \omega_j^T h \p_x u \notag
\\
& = (\omega_j^T B^T \nu_j) h \p_{\omega_j} u. \label{E:Neumann-trans}
\end{align}
Hence
\[
\nu_j^T B B^T h \p_y v|_{F_j} = (\omega_j^T B^T \nu_j) h \p_{\omega_j}
u|_{G_j}.
\]

Plugging these observations in to the formulae for the $I_j$, we get
for $1 \leq j \leq n$ 
\begin{align*}
I_j & = \int_{F_j} ( h \p_{\nu_j} v) (\nu_j^T B B^T h \p \bv )  dS_j
\\
& = \frac{1}{(n-1)! \Vol_{n-1} (G_j)}\int_{G_j} \left((\omega_j^T B^T
\nu_j)^{-1} h \p_{\omega_j} u |_{G_j}\right)  \left((\omega_j^T B^T \nu_j) h \p_{\omega_j}
\bu|_{G_j} \right) d \tS_j \\
& = \frac{1}{(n-1)! \Vol_{n-1} (G_j)}\int_{G_j} | h \p_{\omega_j } u|^2 d \tS_j ,
\end{align*}
where $d \tS_j$ is the induced surface measure on $G_j$.

On the other hand, for $I_0$, we have
\begin{align*}
I_0 & =  \int_{F_0} ( h \p_{\nu_0} v) (\nu_0^T B B^T h \p \bv )  dS_0 \\
& = \frac{n^{1/2}}{(n-1)! \Vol_{n-1} (G_0)} \int_{G_0} | \p_{\omega_0}
u|^2 d \tS_0,
\end{align*}
where $d \tS_0$ is the induced surface measure on $G_0$.

We recall that
\[
\int_{T_0} | v |^2 d y = \frac{1}{n! \Vol_n ( T ) },
\]
so that rearranging we have for each $1 \leq j \leq n$
\[
I_j = \frac{2}{n! \Vol_n ( T ) },
\]
and
\[
I_0 = n^{1/2}\frac{2}{n! \Vol_n ( T ) }.
\]
Rearranging, we have for $0 \leq j \leq n$
\begin{align*}
  \int_{G_j} & | h \p_{\omega_j} u |^2 d \tS_j \\
  & = \frac{2 (n-1)! \Vol_{n-1}(G_j)}{n! \Vol_n (T) } \\
  & = \frac{2  \Vol_{n-1}(G_j)}{n \Vol_n (T) },
\end{align*}
which completes the proof of Theorem \ref{T:main}.

\section{A simple inverse problem}
\label{S:inverse}
The proof of Theorem \ref{T:main} suggests a further question: 
If $u$ solves a constant coefficient eigenfunction equation, does the
Neumann data determine the coefficients?  In fact, in this paper, we
only have information about the norms of the Neumann data, so we
cannot fully answer this question using only this very elementary information.  In fact, in the general case, the
answer is that the norms of the Neumann data do not determine the
coefficients (see Subsection \ref{SS:3d-ex} below).  However, in dimension 2 the norms {\it do} determine
the coefficients.  We will return to this question after a few easier results.

This question is, of course intimately related to posing the standard
Laplacian eigenfunction problem on a different simplex.  Let us pose
it as such in dimension 2.  Let $T \subset \reals^2$ be a triangle
with sides $a,b,c$, with the convention that the length of the sides
are $a,b,c$ respectively.  
Suppose $u$ solves
\begin{equation}
 \label{E:efn-10}
\begin{cases}
  (-h^2 \Delta -1 ) u = 0 \text{ on } T, \\
  u |_{\p T} = 0, \\
  \| u \|_{L^2 (T)}.
\end{cases}
\end{equation}

We have the following Theorem.
\begin{theorem}
  \label{T:tri-inv}
  Suppose $u$ solves \eqref{E:efn-10}, and suppose $N_{a} = \int_a | h
  \p_\nu u |^2 d S$ and similarly for $N_b$ and $N_c$.  Then the three
  quantities $N_a, N_b , N_c$ uniquely determine the triangle $T$ (up
  to reflection).
\end{theorem}
This theorem seems obvious, but in the formulae for the $N_a$, $N_b$,
and $N_c$, there is both the length of the side {\it and} the area of
the triangle.  The proof is by scaling.

\begin{proof}

  Suppose we have another triangle $T_1$ with the same Neumann data
  norms.  Let $a_1, b_1, c_1$ denote the three sides of $T_1$, again
  with the convention that $a_1, b_1, c_1$ denote also the length of
  the sides.  We know that the Neumann data relates the lengths of the
  sides to the area of the triangle.  We have
  \[
  N_a = \frac{a}{\text{Area}(T)},
  \]
  and similarly for $b,c$.  On the other hand, we also have
  \[
  N_a = \frac{ a_1}{\text{Area}(T_1)},
  \]
  and similarly for  $b_1, c_1$.  Equating these quantities, we have
  \[
  \frac{a}{a_1} = \frac{\text{Area}(T)}{\text{Area}(T_1)},
  \]
  and similarly
  \[
  \frac{b}{b_1} = \frac{c}{c_1} =
  \frac{\text{Area}(T)}{\text{Area}(T_1)}.
  \]
  This means that the side lengths of $T_1$ are all scalar multiples
  of the corresponding sides on $T$ with the same scalar.  Hence $T_1$
  is similar to $T$.  Let
  \[
  \lambda = \frac{\text{Area}(T)}{\text{Area}(T_1)}.
  \]
  On the one hand, this implies that
  \begin{equation}
    \label{E:T-T1}
    {\text{Area}(T)} = \lambda {\text{Area}(T_1)}.
    \end{equation}
    On the other hand, we have
    \begin{equation}
      \label{E:a-a1}
    a = \lambda a_1
    \end{equation}
    and similarly
    \begin{equation}
      \label{E:b-b1}
    b = \lambda b_1, \,\,\, c = \lambda c_1.
    \end{equation}
    As the lengths scale linearly, the area scales quadratically.
    That is, \eqref{E:a-a1} and \eqref{E:b-b1} imply that
    \[
    \text{Area}(T) = \lambda^2 \text{Area}(T_1).
    \]
    Hence combining with \eqref{E:T-T1}, we have $\lambda^2 = \lambda$,
    so that $\lambda = 1$.  This means precisely that $T = T_1$ (up to reflection).

\end{proof}

We now consider the question of determining the coefficients of a
constant coefficient elliptic operator on the standard 2-simplex.  Let
$B$ be a non-degenerate $2 \times 2$ matrix, and let $\Gamma = B
B^T$.  Consider $P = - \Gamma_{ij} h \p_{x_i} h \p_{x_j}$ be the
associated positive definite elliptic operator.  Our next result is
that the semi-classical Neumann data  uniquely determines the operator $P$.
Interestingly, this does not determine the matrix $B$ (see Remark \ref{R:non-unique}).

\begin{theorem}
  \label{T:Gamma-I}
  Let $B$ be a non-degenerate $2 \times 2$ matrix and $\Gamma = B
  B^T$.  Let $P = - \Gamma_{ij} h \p_{x_i} h \p_{x_j}$.  Let $T_0$ be
  the standard triangle in $\reals^2$ generated by the vectors $(1,0)$
  and $(0,1)$.  Suppose $u$ solves the eigenfunction problem
  \[
  \begin{cases}
    P u = u \text{ in } T_0,
    \\
    u |_{\p T_0} = 0, \\
    \| u \|_{L^2( T_0)} = 1.
  \end{cases}
  \]
  Let $F_1$ and $F_2$ denote the sides of length $1$ and $F_0$ the
  hypotenuse of length $\sqrt{2}$.  Then the norms
  \[
  \| h \p_{\nu } u
  \|_{L^2(F_1)}^2 ,\, \| h \p_{\nu } u
  \|_{L^2(F_2)}^2, \, \text{and }\| h \p_{\nu } u
  \|_{L^2(F_0)}^2  
  \]
  uniquely determine $\Gamma$.
 
\end{theorem}

\begin{remark}
  We pause to remark that in the statement of the theorem is buried a
  rather astounding fact: the norms of the (semi-classical) Neumann data of {\it any
    single} eigenfunction determine $\Gamma$.  Of course this requires
  some knowledge also about the spectrum.  In other words, if one
  eigenvalue and corresponding eigenfunction's Neumann mass is known,
  then $\Gamma$ is  uniquely
  determined.

\end{remark}

\begin{remark}
It is also very interesting that the proof in fact computes the
entries of $\Gamma$ explicitly in terms of the Neumann data norms.  
Indeed, if we label
 \[
  J_1 = \| h \p_{\nu } u
  \|_{L^2(F_1)}^2 ,\, J_2 = \| h \p_{\nu } u
  \|_{L^2(F_2)}^2, \]
  and
  \[ J_0 = \| h \p_{\nu } u
  \|_{L^2(F_0)}^2 , 
  \]
  and we write $\Gamma = (\Gamma)_{jk}$, we have
  \[
  \Gamma_{11} = \frac{2}{J_1},
  \]
  \[
  \Gamma_{22} = \frac{2}{J_2},
  \]
  and
  \[
  \Gamma_{12} = \Gamma_{21} = \frac{2 \sqrt{2}}{J_0} - \frac{1}{J_1} -
  \frac{1}{J_2}.
  \]

  In particular, if $J_1 = J_2 = 2$ and $J_0 = 2 \sqrt{2}$, we have
  $\Gamma = I$ as expected (since each $J_j$ is twice the length of
  the sides, which is the length of the side divided by the area of
  the triangle).

  \end{remark}

First we write a Lemma giving yet another way of computing the Neumann
data mass.  We state this Lemma in any dimension.

\begin{lemma}
  \label{L:other-N}
Let $B$ be a non-degenerate $n \times n$ matrix, $\Gamma = B B^T$, and
\[
P = - \Gamma_{ij} h \p_{x_i} h \p_{x_j}.
\]
Let $T_0$ be the standard
simplex in $\reals^n$ with faces $F_0, F_1, \ldots , F_n$ in the
notation of earlier in this paper.  Suppose $u$ solves the eigenfunction problem
\begin{equation}
  \label{E:ef-111}
 \begin{cases}
    P u = u \text{ in } T_0,
    \\
    u |_{\p T_0} = 0, \\
    \| u \|_{L^2( T_0)} = 1.
  \end{cases}
 \end{equation}
 Then on each face $F_j$, $0 \leq j \leq n$, we have
 \[
 \int_{F_j} ( h \p_{\nu_j} u ) (\nu_j^T B B^T h \p_x \bu ) dS_j = | B^T \nu_j |^2
 \int_{F_j} | h \p_{\nu_j} u |^2 dS_j,
 \]
where $dS_j$ is the induced surface measure on $F_j$ as usual.

\end{lemma}

Note that this is a different way of computing this quantity than in \eqref{E:Neumann-trans}.

\begin{proof}

We  observe that on $F_j$, for $1 \leq j \leq n$, $h \p_{x_\ell} \bu = 0$ for
$\ell \neq j$, so that 
\[
h \p_x u = e_jh \p_{x_j} u = \nu_j h \p_{\nu_j} u.
\]
Similarly, for $j = 0$, we have on $F_0$
\begin{align*}
h \p u & = 
\left[ \begin{array}{c}
    h \p_{y_1} \\
    \vdots \\
    h \p_{y_n}
  \end{array}
  \right] u \\& = 
n^{-1/2} \left[ \begin{array}{c}
    1 \\
    \vdots \\
    1
  \end{array}
  \right] h \p_{\nu_0} u     \\
& = \nu_0 h \p_{\nu_0} u 
.
\end{align*}


Then on each face $F_j$ with normal $\nu_j$, we have
\begin{align*}
  \nu_j^T B B^T h \p \bu & = \nu_j^T B B^T \nu_j h \p_{\nu_j} \bu \\
  & = ( B^T \nu_j)^T ( B^T \nu_j)  h \p_{\nu_j} \bu \\
  & = | B^T \nu_j  |^2 h \p_{\nu_j} \bu.
\end{align*}
Hence on each face $F_j$, we have
 \[
 \int_{F_j} ( h \p_{\nu_j} u ) (\nu_j^T B B^T h \p \bu ) dS_j = | B^T \nu_j |^2
 \int_{F_j} | h \p_{\nu_j} u |^2 dS_j.
 \]
 This completes the proof.
 \end{proof}

\begin{proof}[Proof of Theorem \ref{T:Gamma-I}]

The proof proceeds by using an eigenvector diagonalization argument.
It is interesting that, although the argument uses the {\it existence}
of eigenvalues/vectors of $\Gamma$, we do not need to know them.

Let $v_1, v_2$ be orthonormal eigenvectors for $\Gamma$.  Since
$\Gamma = B B^T$ is positive definite, write $\lambda_1^2,
\lambda_2^2$ for the eigenvalues of $\Gamma$ so that $\Gamma v_j =
\lambda^2_j v_j$ for $j = 1,2$.  Let
\[
L = \left( \begin{array}{cc} \vline & \vline \\
  v_1 & v_2 \\
  \vline & \vline
\end{array}
\right),
\]
so that (since $L$ is orthogonal),
\[
L^T \Gamma L = \left( \begin{array}{cc} \lambda_1^2 & 0 \\ 0 &
  \lambda_2^2 \end{array} \right).
\]
Let us denote
\[
G = L \left( \begin{array}{cc} \lambda_1 & 0 \\ 0 &
  \lambda_2 \end{array} \right),
\]
so that $G G^T = \Gamma$.

We now change variables using the matrix $G$.  Let $T_1$ denote the
triangle spanned by the new coordinates $v_1, v_2$.  Rescaling in each
variable $v_j \mapsto \lambda_j^{-1} v_j$ gives a new triangle $T$.
Let $w(x) = u(Gx)$, so that
\begin{equation}
  \label{E:Neumann-unique}
\int_T | w |^2 dV = \int_T | u(Gx) |^2 dV = |G|^{-1} \int_{T_0} | u
|^2 dV = \frac{1}{\lambda_1 \lambda_2}.
\end{equation}
We also have $-h^2 \Delta w = w$ on $T_0$, so we can use the same
commutator argument as above to compute the mass of the Neumann data.
For $j = 1,2$, let $I_j = \int_{\lambda_j^{-1} v_j } | h \p_\nu w |^2
dS$, and $I_0 = \int_{H} | h \p_\nu w |^2 dS$ be the Neumann mass of
the function $w$ on the legs spanned by the $\lambda_j^{-1} v_j$ and
the hypotenuse $H$.  Using Theorem \ref{T:main} and
\eqref{E:Neumann-unique}, we have for $j = 1,2$ 
\begin{align*}
  I_j & = 
\left( \frac{1}{\lambda_1 \lambda_2} \right) \left( \frac{
  \text{length of } \lambda_j^{-1} v_j }{\text{area}(T)} \right) \\
& = \left( \frac{1}{\lambda_1 \lambda_2} \right) \left( \frac{
  \lambda_j^{-1} }{(\lambda_1^{-1} \lambda_2^{-1} / 2)} \right) \\
& = \frac{2}{\lambda_j}.
\end{align*}
Further,
\[
I_0 = 2 ( \lambda_1^{-2} + \lambda_2^{-2} )^{1/2}.
\]

For $j = 0,1,2$, let
\[
J_j = \int_{F_j} | h \p_\nu u |^2 dS
\]
be the Neumann mass of the original eigenfunction $u$ on the faces of $T_0$.  These
are the quantities we are assuming we know.

Using Lemma \ref{L:other-N}, that means
\begin{align}
J_j & = \frac{1}{|G^T \nu |^2 } \int_{F_j} ( h \p_\nu u ) (\nu^T G G^T h
\nabla \bu ) dS \notag \\
& = \left( \frac{1}{|G^T \nu |^2 } \right) \left( \frac{ \text{length
    of }F_j}{\text{length of }\lambda_j^{-1} v_j} \right) I_j \notag \\
& = \left( \frac{1}{|G^T \nu |^2 } \right) \lambda_j \left(
\frac{2}{\lambda_j} \right) \notag \\
& = \frac{2}{| G^T \nu |^2 } \label{E:J-to-G}
\end{align}
for $j = 1,2$, and
\[
J_0 = \frac{2 \sqrt{2}}{| G^T \nu |^2 }.
\]
We pause momentarily to recall that the normal vectors $\nu$ in the
above expressions are the normals to the original faces $F_j$, $j =
0,1,2$ on the standard triangle $T_0$.


  Recall that $\nu_1 = (-1,0)$, $\nu_2 = (0,-1)$ and $\nu_0 =
  (\sqrt{2})^{-1} ( 1,1 )$, which will help us determine the matrix
  $\Gamma$.

  Write
  \[
  \Gamma = \left( \begin{array}{cc} \Gamma_{11} & \Gamma_{12} \\
    \Gamma_{21} & \Gamma_{22} \end{array} \right).
  \]
  As $\Gamma$ is symmetric, we have $\Gamma_{12} = \Gamma_{21}$, so we
  only need to determine the three numbers $\Gamma_{11}, \Gamma_{12},$
  and $\Gamma_{22}$.

  The quantities we need to examine are all of the form $| G^T \nu_j
  |^2$, which we rewrite:
  \begin{align*}
    | G^T \nu_j |^2 & = (G^T \nu_j)^T (G^T \nu_j) \\
    & = \nu_j^T G G^T \nu_j \\
    & = \nu_j^T \Gamma \nu_j.
  \end{align*}
  
  Plugging in the $\nu_j$, $j = 0 , 1 , 2$, we have:
  \begin{align*}
    \nu_1^T \Gamma \nu_1 & =
    (-1, 0 ) \Gamma \left( \begin{array}{c} -1 \\ 0 \end{array} \right) \\
    & =    \Gamma_{11} ,
  \end{align*}
  and similarly
  \[
  \nu_2^T \Gamma \nu_2 = \Gamma_{22}.
  \]
  For $\nu_0$, we get information about the off diagonal terms as
  well:
  \begin{align}
    \nu_0^T \Gamma \nu_0 & = \frac{1}{2}(1,1) \Gamma
    \left( \begin{array}{c} 1 \\ 1 \end{array} \right)  \notag \\
      & = \frac{1}{2} (1,1) \left( \begin{array}{c} \Gamma_{11} +
        \Gamma_{12} \notag \\
        \Gamma_{21} + \Gamma_{22} \end{array} \right) \notag \\
    & = \frac{1}{2} ( \Gamma_{11} + \Gamma_{12} + \Gamma_{21} +
    \Gamma_{22} ) \notag \\
    & = \frac{1}{2} ( \Gamma_{11} + 2 \Gamma_{12} + \Gamma_{22} )
    \label{E:J-0-101}
  \end{align}
  again due to $\Gamma$ being symmetric.

  Returning now to \eqref{E:J-to-G}, we have for $j = 1,2$
  \begin{align*}
    J_j & = 
    \frac{2}{| G^T \nu_j |^2 } \\
    & = \frac{2}{\Gamma_{jj}}.
  \end{align*}
  Hence 
  \[
  \Gamma_{11} = \frac{2}{J_1}
  \]
  and similarly for $\Gamma_{22}$.  For $\Gamma_{12}$, we appeal to
  equation \eqref{E:J-0-101} to get 
  \begin{align*}
    J_0 & = 
    \frac{2 \sqrt{2}}{| G^T \nu_0 |^2 } \\
    & = \frac{2 \sqrt{2}}{\frac{1}{2} ( \Gamma_{11} + 2 \Gamma_{12} +
      \Gamma_{22} )} \\
     & = \frac{4 \sqrt{2}}{ \Gamma_{11} + 2 \Gamma_{12} + \Gamma_{22}}.
  \end{align*}
  Rearranging, we have
  \[
  \Gamma_{11} + 2 \Gamma_{12} + \Gamma_{22} = \frac{ 4 \sqrt{2}}{J_0},
  \]
  so that
  solving for $\Gamma_{12}$, we have
  \[
  \Gamma_{12} = \frac{2 \sqrt{2}}{J_0} - \frac{1}{2} ( \Gamma_{11} + \Gamma_{22}
  ).
  \]  
Plugging in the known values of $\Gamma_{11}$ and $\Gamma_{22}$, we
have
 \begin{align*}
  \Gamma_{12} & = \frac{2 \sqrt{2}}{J_0} - \frac{1}{2} \left(
  \frac{2}{J_1}  + \frac{2}{J_2} \right) 
   = \frac{2 \sqrt{2}}{J_0} - \frac{1}{J_1} - \frac{1}{J_2}.
 \end{align*}

This gives the $\Gamma_{jk}$ in terms of the known quantities $J_1,$
$J_2, $ and $J_0$, completing the proof.
 \end{proof}





\begin{remark}
  \label{R:non-unique}
It is interesting to note that the proof of Theorem \ref{T:Gamma-I}
does not uniquely determine the matrix $B$, due to rotational invariance.  Indeed, if
\[
B = \left( \begin{array}{cc} a & b \\ c & d \end{array} \right)
\]
with $a = c = d = 2^{-1/2}$ and $b = - 2^{-1/2}$, then we still have
$a^2 + c^2 = b^2 + d^2 = 1$, and $ac + bd = 0$.  Note, however, that
$B B^T = I$ in this case as well.

  \end{remark}

\subsection{Dimension 3: an example}
\label{SS:3d-ex}
The result in Theorem \ref{T:Gamma-I} is false in higher dimensions,
even for small perturbations of $I$.  
Let $T_0$ be the standard simplex in $\reals^3$, $B$ be a $3 \times 3$
non-degenerate matrix, $\Gamma = B B^T$, and $P = - \Gamma_{ij} h
\p_{x_i} h \p_{x_j}$.  Suppose $u$ solves the eigenfunction problem
\eqref{E:ef-111}.  Lemma \ref{L:other-N} still applies, with $\nu_j =
- e_j$ for $1 \leq j \leq 3$ and $\nu_0 = 3^{-1/2} ( 1,1,1)$.   For
$0< \epsilon<1$, define the matrix $B$ by 
\[
B^T = \left( \begin{array}{ccc}
  a & 0 & 0 \\
  d & (1 - \epsilon^2)^{1/2} & \epsilon \\
  \epsilon & \epsilon & (1 - \epsilon^2)^{1/2}
\end{array} \right),
\]
where
\[
d = \frac{-3 \epsilon ( 1 - \epsilon^2)^{1/2} - \epsilon^2}{(1 -
  \epsilon^2)^{1/2} + \epsilon }
\]
and
\[
a = (1 - d^2 - \epsilon^2 )^{1/2}.
\]
Observe that $B = I + \O(\epsilon)$ and 
satisfies
\[
| B^T e_1 |^2 = a^2 + d^2 + \epsilon^2 = 1,
\]
\[
| B^T e_2 |^2 = (1 - \epsilon^2) + \epsilon^2 = 1,
\]
\[
| B^T e_3 |^2 = \epsilon^2 + (1 - \epsilon^2) = 1,
\]
and
\begin{align*}
| B^T (1,1,1)^T |^2 & = a^2 + (d + (1 - \epsilon^2)^{1/2} +
\epsilon)^2 + ( 2 \epsilon + (1 - \epsilon^2 )^{1/2} )^2 \\
& = a^2 + d^2 + (1 - \epsilon^2) + \epsilon^2 + 2 d (1 -
\epsilon^2)^{1/2} + 2 d \epsilon + 2 \epsilon ( 1 - \epsilon^2)^{1/2}
\\
& \quad + 4 \epsilon^2 + (1 - \epsilon^2) + 4 \epsilon (1 - \epsilon^2
)^{1/2} \\
& = (1 - \epsilon^2 ) + (1 - \epsilon^2) + \epsilon^2 + 2d ( (1 -
\epsilon^2)^{1/2} + \epsilon) \\
& \quad + 2 \epsilon ( 1 - \epsilon^2)^{1/2} + 4 \epsilon^2 + (1 -
\epsilon^2) + 4 \epsilon ( 1 - \epsilon^2)^{1/2} \\
& = 2 - \epsilon^2 - 6 \epsilon ( 1 - \epsilon^2)^{1/2} - 2 \epsilon^2
+ 6 \epsilon ( 1 - \epsilon^2)^{1/2} + 1 + 3 \epsilon^2 \\
& = 3.
\end{align*}
These are the same values one gets from $B = I = \Gamma$, 
however $B B^T \neq I$, so these 4 numbers do not determine $\Gamma$.




\bibliographystyle{alpha}
\bibliography{HC-bib}

\end{document}